\documentclass[12pt]{amsart}
\usepackage{latexsym,fancyhdr,amssymb,color,amsmath,amsthm,graphicx,listings,comment}
\usepackage[section]{placeins}
\newtheorem{thm}{Theorem}
\newtheorem{lemma}{Lemma}
\newtheorem{propo}{Proposition}
\newtheorem{coro}{Corollary}
\let\paragraph\subsection

\title{The Three Tree Theorem}
\author{Oliver Knill}
\date{September 4, 2023}
\address{Department of Mathematics \\ Harvard University \\ Cambridge, MA, 02138 }
\subjclass{}

\keywords{Graph theory, Arboricity, Chromatic number, Planar graphs, 4-coloring}

\begin{document}
\maketitle

\begin{abstract}
We prove that every 2-sphere graph different from a prism can be vertex 4-colored in such a way 
that all Kempe chains are forests. This implies the following "three tree theorem": 
the arboricity of a discrete 2-sphere is 3. Moreover, the 
three trees can be chosen so that each hits every triangle. 
A consequence is a result of an exercise in the book of Bondy and Murty
based on work of A. Frank, A. Gyarfas and C. Nash-Williams: the arboricity of 
a planar graph is less or equal than 3. 
\end{abstract}

\section{Preliminaries}

\paragraph{}
A finite simple graph $G=(V,E)$ is called a {\bf 2-manifold}, if every {\bf unit 
sphere} $S(v)$, the sub-graph of $G$ induced by $\{ w \in V \; | \;  (v,w) \in E \}$ 
is a {\bf 1-sphere}, a cyclic graph with $4$ or more elements. Because every edge in a $2$-manifold 
bounds two triangles and every triangle is surrounded by three edges, the set 
$T$ of triangles satisfies the {\bf Dehn-Sommerville relation} $3|T| = 2|E|$.
Since $G$ is $K_4$-free, the {\bf Euler characteristic} of $G$ is 
$\chi(G) = |V|-|E|+|T|$, where $T$ is the set of {\bf triangles} $K_3$, 
$E$ the set of {\bf edges} $K_2$ and $V$ the set of {\bf vertices} $K_1$ in $G$.
If a $2$-manifold $G$ is connected and $\chi(G)=2$, it is called a {\bf $2$-sphere}.
The class of $2$-spheres together with $K_4$ are known to agree with the class of 
maximally planar, $4$-connected graphs.

\paragraph{}
$G$ is {\bf contractible} if there is $v \in V$ such that both the unit sphere graph
$S(v)$ and the graph $G \setminus v$ induced by $V \setminus \{v\}$ are contractible. 
This inductive definition starts by assuming that $1=K_1$ is contractible.
The {\bf zero graph} $0$, the empty graph, is declared to be the {\bf $(-1)$-sphere}. 
For $d \geq 0$, a {\bf d-sphere} $G$ is is a $d$-manifold for which 
$G \setminus v$ is contractible for some vertex $v$. A {\bf $d$-manifold} is a graph for
which every unit sphere is a $(d-1)$-sphere. If $G$ is a $d$-sphere, 
the two contractible sets, $G \setminus v$ and the unit ball $B(v)$ together 
cover $G$, so that its {\bf category} is $2$ as in the continuum. 

\paragraph{}
Using induction one can show that a contractible graph satisfies 
$\chi(G)=1$ and a sphere satisfies the {\bf Euler gem formula} 
$\chi(G)=1+(-1)^d$ which for $d=2$ is $2$. 
The Euler-Poincar\'e formula $\sum_{k=0}^d (-1)^k f_k=\sum_{k=0}^d (-1)^k b_k$ holds
for a general {\bf finite abstract simplicial complex}, where $f_k$ is
the number of $K_{k+1}$ sub-graphs of $G$ and $b_k$ is a {\bf Betti number}, the nullity of 
the $k$-th Hodge matrix block. In the connected $2$-dimensional case, the Euler-Poincar\'e formula is
$|V|-|E|+|T|=1-b_1+b_2$, implying  $\chi(G) \leq 2$. Indeed, the identity $\chi(G)=2$ forces 
$b_1=0$ and $b_2=1$ so that $G$ is orientable of genus $0$, justifying the 
characterization of 2-spheres using the Euler characteristic functional alone, 
among connected 2-manifolds. 

\paragraph{}
{\bf Examples:} 1) The {\bf octahedron} and {\bf icosahedron} are $2$-spheres. 
2) The suspension of a 1-sphere $C_n$ is a {\bf prism graph}. It is a 2-sphere for 
$n \geq 4$. For $n=4$, it is the octahedron. Prisms are the only 2-dimensional 
non-primes in the {\bf Zykov monoid} \cite{Zykov} of all spheres. No other 2-sphere can be written 
as $A \oplus B$ with non-zero $A,B$. 
3) The {\bf tetrahedron} $K_4$ is not a 2-manifold because $S(v)=K_3$. It is 
{\bf 4-connected} and maximally planar although. One could look at the 2-dimensional
skeleton complex of $K_4$ which is a 2-dimensional simplicial complex qualifying for a
sphere. Within graph theory, where the simplicial complexes are the Whitney complexes, 
it is not as sphere. 4) The {\bf triakis-icosahedron}  
is a triangulation of the Euclidean 2-sphere $\{ |x|=1, x \in \mathbb{R}^3 \}$ but
as it is having $K_4$ sub-graphs, it is not a 2-manifold. Also here, one could look at the
2-skeleton complex and get as a geometric realization a triangulation of a regular sphere.
We do not allow degree 3 vertices as this produces tetrahedra which are 3-dimensional. 
5) The {\bf cube} and {\bf dodecahedron} both have dimension $1$ and are not 2-manifolds. 
Also here, one would have to transcend the frame work to include it as a 2-sphere. One 
would look at CW complexes, meaning a discrete cell complex,
where the quadrangles or pentagons are included as cells. There had been a lot of confusion 
about defining polyhedra and polytopes \cite{lakatos,Richeson}. 
6) The {\bf Barycentric refinement} $G_1=(V_1,E_1)$ of a 2-sphere $G$ is a 2-sphere. 
It has as vertex set $V_1$, the set of complete sub-graphs of $G$ and 
the edge set $E_1=\{ (x,y), x \subset y, \; {\rm or} \;  y \subset x\}$. We currently
are under the impression that the arboricity in the barycentric limit could give an 
upper bound for the arboricity of all d-spheres. 
7) An {\bf edge refinement} $G_e=(V_e,E_e)$ of a 2-sphere $G$ with edge $e=(a,b)$ is a 2-sphere.
One has $V_e = V \cup \{e\}, E_e=E \cup \{(v,a),(v,b),(v,c),v(d)\}$ with $\{c,d\}=S(a) \cap S(b)$.
8) The reverse of 7), {\bf edge collapse} can be applied to a degree 4 vertex $S(v)$ 
for which all $w \in S(v)$ satisfy ${\rm deg}_G(w) \geq 6$. The result is then again a 2-sphere. 
9) A {\bf vertex refinement} $G_{a,b}$ picks two non-adjacent points $\{a,b\}$ in $S(v)$ 
and takes $V_{a,b}=V \cup \{w\}$ and $E_{a,b} = E \cup \{ (v,w),(a,w),(b,w) \}$. 
10) The reverse of 9), the {\bf kite collapse}, can be applied to an embedded 
{\bf kite} $K=K_2 +2$ in $G$ if the dis-connected vertex pair $\{a,b\}$ in $K$ 
both have degrees $5$ or larger. 

\begin{figure}[!htpb]
\scalebox{0.5}{\includegraphics{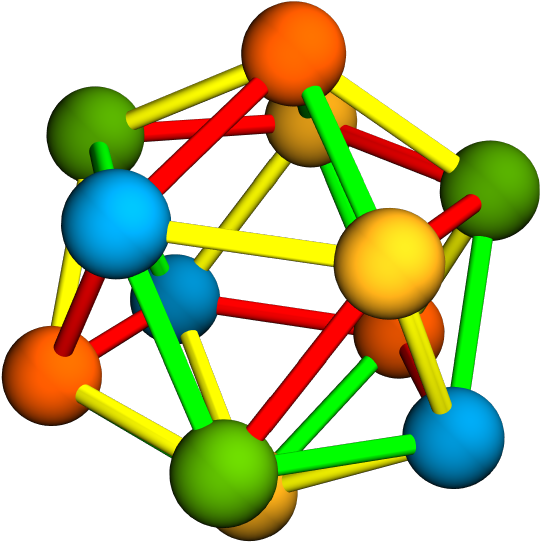}}
\scalebox{0.5}{\includegraphics{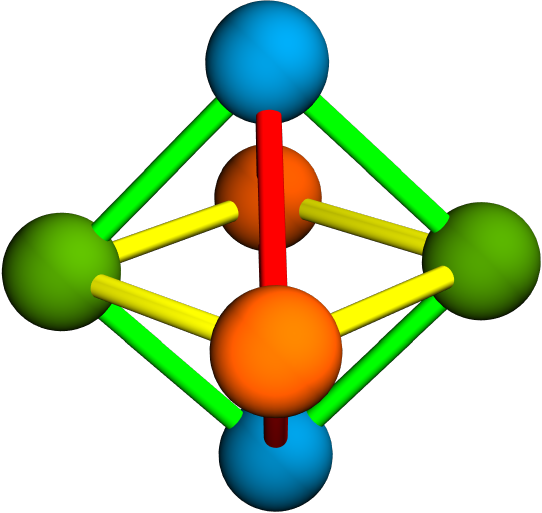}}
\caption{
The icosahedron $G$ to the left with a 4-coloring produced the 3 Kempe 
forests $A,B,C$. The octahedron to the right has chromatic number 3.
But any 3 or 4-coloring $f$ of $G$ leads to a coloring $F$ which produces
a closed Kempe loop. Any prism graph has arboricity 3 but in the prism case,
the three trees can never be Kempe trees. Prism graphs
are also the only 2-spheres which are a join of two lower
dimensional spheres. }
\end{figure}

\paragraph{}
The {\bf arboricity} of a graph $G$ is defined as 
the minimal number of forests that partition the graph \cite{BM}. 
A {\bf forest} is a triangle-free graph for which every connected component is 
contractible. The empty graph $0$ is not a tree and has arboricity $0$, the graph $1=K_1$ as well as any
0-dimensional manifold (a graph without edges) has arboricity $1$ because it is
a forest. For a {\bf connected} graph $G$, the arboricity of $G$ agrees with the 
minimal number of sub trees which cover the graph.
Proof: a finite set of forests $F_1,\dots,F_k$ partitioning $G$ 
can in a connected $G$ be completed to a set of covering trees $T_k$ of $G$ if $G$ is connected. 
Conversely, every collection of covering trees $T_1,\dots,T_k$ 
can be morphed into a collection of disjoint partitioning 
forests $F_1,\dots,F_k$, where $F_i$ are obtained from $T_i$ by 
deleting edges $w$ that are already covered by other trees. We get back to this in 
an appendix. 

\paragraph{}
The {\bf Nash-Williams theorem} \cite{NashWilliams,CMWZZ,HararyGraphTheory} 
identifies the arboricity of a positive dimensional graph as the least integer
$k$ such that $|E_H| \leq k (|V_H|-1)$ for all sub-graphs $H=(V_H,E_H)$
and the understanding is that the vertex set $V_H$ {\bf generates} the graph $H$,
meaning that if $v,w$ are two nodes in $H$ and the connection $(v,w)$ is
in $E$ then also $(v,w)$ is in $H$. 
The understanding is also that for $H=K_1$ where any $k$ would work for Nash-Williams
the arboricity is $1$ and that for $H=0$, the empty graph, the arboricity is $0$.
In particular, $|E|/(|V|-1)$ is always a lower bound for the arboricity for
a connected graph $G=(V,E)$. Since any 4-connected planar graph 
different from $K_4$ is a sub-graph of a 2-sphere, the arboricity of a 
4-connected planar graph is bounded by the maximal arboricity
that is possible of 2-spheres. As we will show here, 2-spheres have arboricity $3$ so that all
planar graphs will have arboricity $\leq 3$, a result which appears as an exercise 
21.4.6 in \cite{BM} based on work of A. Frank, A. Gyarfas and C. Nash-Williams.
The text \cite{Ruohonen} sees it as a consequence of the {\bf Edmond Covering Theorem} in 
matroid theory. 

\paragraph{}
{\bf Examples}:
1) The arboricity of a cyclic graph 
$C_n=(V,E)=( \{0,1, \dots, n-1\}, \{ (k,k-1) \; {\rm mod} \; n, 1, \leq n\})$ 
with $n \geq 4$ is $2$. It is larger or equal to $|E|/(|V|-1)=n/(n-1)>1$ and two linear trees like 
$T_1=(\{0,1\}\{ (01) \} )$ and $T_2=C_n \setminus K_2 = (V,E \setminus \{ (01) \})$ cover $C_n$. 
That the arboricity of $C_n$ is $2$ also follows from the fact that $C_n$ itself is not a tree.
2) The arboricity of a {\bf figure 8 graph} is $2$ too. It can be partitioned into 2 forests $F_1,F_2$,
where $F_1$ is a star graph with 5 vertices and $F_2$ is a forest consisting of 2 trees. 
From the two forests, one can get a tree cover by taking $T_1=F_1$ 
by taking $T_2$ the two trees $F_2$ with a path connecting them. This illustrates the above
switch from a {\bf forest partition} to a {\bf tree cover} which works in the connected graph case. 

\paragraph{}
3) The smallest {\bf $2$-torus} that is a manifold is obtained by triangulating a $5 \times 5$ 
grid and identifying the left-right and bottom-top boundaries to get 16 vertices and $2*16=32$ triangles.
It has the f-vector $(|V|,|E|,|T|)=(16,48,32)$ so that $|E|/(|V|-1)=48/15>3$ and an explicit cover with 
$4$ forests shows the arboricity is $4$. This is larger than the
number $3$ obtained in the Barycentric limit. It is still not excluded that the arboricity 
of a $d$-manifold is always smaller or equal than the smallest integer larger 
than the Barycentric limit number $c_d$ to be defined later and which is $c_2=3$ 
in two dimensions. A conjecture of Albertson and Stromquist states that all 2-manifolds have cromatic number
$\leq 5$ \cite{AlbertsonStromquist}. 
4) An {\bf octahedron} with $|V|=6,|E|=12,|T|=8$ has arboricity $3$,
because $|E|/(|V|-1)=12/5 >2$ and because there are three 
spanning trees can cover it: start with two star graphs and a circular graph
partitioning the edge set, then switch one of the equator colors. 
5) There is a {\bf projective plane} of chromatic number $5$. Its $f$-vector is
$f=(15, 42, 28)$. The arboricity is 3.  

\begin{figure}[!htpb]
\scalebox{0.5}{\includegraphics{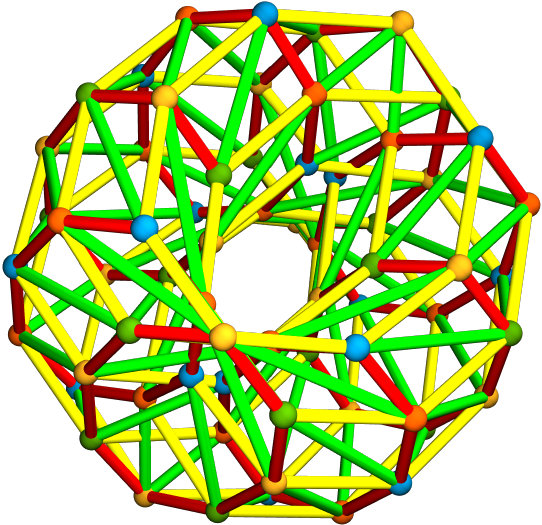}}
\scalebox{0.5}{\includegraphics{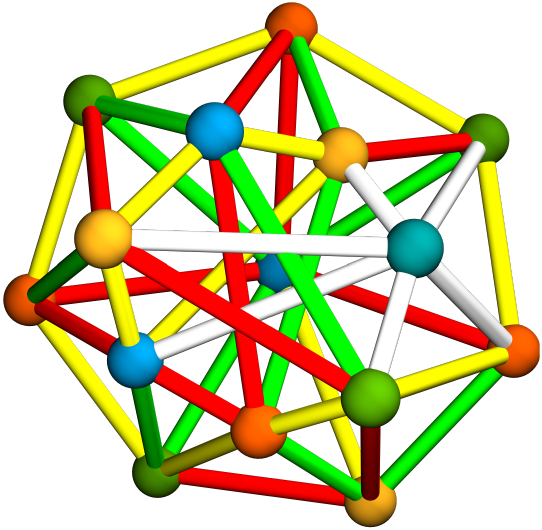}}
\caption{
A torus graph which has a $4$-coloring $f$ and arboricity $4$. 
A 2-torus of chromatic number 4 and arboricity 4 
always must have a closed Kempe loop as without a Kempe loop, we would
get 3 Kempe forests and so arboricity 3. 
The figure shows an example of a colored 2-torus. There are lots of 
closed Kempe chains. To the right we see a projective plane with a 5
coloring. The Kempe chain construction does not work any more. In the picture
we colored the Kempe chains of the first 4 colors. 
}
\end{figure}

\section{The theorem}

\paragraph{} 
The proof of the following theorem makes heavy use of the 
{\bf 4 color theorem} \cite{StromquistPlanar,AppelHaken1}. 

\begin{thm}[Three Tree Theorem]
The arboricity of any 2-sphere is $3$.
\end{thm}

\paragraph{}
The arboricity of a sub-graph $H$ of $G$ is smaller or equal 
than the arboricity of $G$ and any 4-connected planar graph is contained
in a maximally planar 4-connected graph and so a 2-sphere. 
For non-4-connected graphs, cover the 4-connected components forests and 
glue the forests together. The following theorem is mentioned in \cite{BM}
as Exercise 21.4.6 based on results of Frank, Frank and Gyarfas and using 
the Nash-William theorem. 

\begin{coro}[Arboricity of planar graphs]
The arboricity of any planar graph is $\leq 3$. 
\end{coro}

\paragraph{}
We prove the stronger result, telling that any 2-sphere different from
a prism can be neatly covered with forests. A {\bf neat forest cover}
is a cover by 3 forests such that every triangle hits each of the three 
forests in exactly one point. 
If $G$ is a $2$-sphere, define the {\bf upper line graph} as the graph which has
the edges of $G$ as vertices and where two are connected if they are in a 
common triangle. This is a 4-regular graph so that
the chromatic number is less or equal to $4$. We will see that the chromatic
number of the upper line graph is $\leq 3$, the coloring given by the Kempe chains. 
Just to compare, the {\bf lower line graph} or simply {\bf line graph}
connects two edges, if they intersect. 

\paragraph{}
The proof of the theorem follows from two propositions. 
The first gives a lower bound: 

\begin{propo}
The arboricity of any $2$-manifold is larger than $2$. 
\end{propo}
\begin{proof} 
The {\bf Euler-Poincar\'e formula} tells 
$\chi(G) = b_0-b_1+b_2 = 1+b_2 - b_1$ so that $\chi(G)=2-b_1$ or 
$\chi(G)=1-b_1$, pending on the orientability of $G$.
Use the Euler gem formula $|V|-|E|+|T| \leq 2$ \cite{Richeson} as well as the 
{\bf Dehn-Sommerville relation} $3|T|=2|E|$ to get $|V|-|E|/3 \leq 2$ or 
$|E| \geq 3|V|-6$. If there were two spanning trees covering all the 
edges, then the {\bf Nash-Williams formula} gives the bound $|E| \leq 2|V|-2$. 
This is incompatible with $|E|=3|V|-6$ for $|V|>4$ as plotting the functions
$|E|=2|V|-2,|E|=3|V|-6$ shows.  
\end{proof} 

\paragraph{}
The second proposition gives an upper bound: 

\begin{propo}[3 trees suffice]
A $2$-sphere $G$ can be covered with $3$ trees.
\end{propo} 

\paragraph{}
This is exercise 21.4.6 in \cite{BM} based on work of A. Frank, 
A. Gyarfas and C. Nash-Williams. 
We will prove something slightly stronger in that
we also can assure that all of the three trees intersects
any of the triangles, provided that the graph $G$ is not a prism. 

\paragraph{}
We were not aware of previous work on the arboricity of planar graph
before the google AI agent "Bard"
told us about it in a personal communication. Since the agent had
hallucinated during that chat with other claims like that the arboricity of 
a cyclic graph $C_n$ is $1$, we did not take it seriously, but it prompted
us to hit the literature again and let us to find the
exercise in \cite{BM}. We are not aware of a publication proving that 
planar graphs have arboricity $\leq 3$. We will come back to this. 

\paragraph{}
We will establish the proposition standing on the shoulders of the
{\bf 4-color theorem}. There is a {\bf discrete contraction map} to
shrink colored spheres. It uses {\bf edge collapses} and {\bf kite collapses}.
It can be seen as a "discrete heat flow"
because edge collapse removes curvature $1/3$ vertices and kite
collapses reduce and merge vertex degrees which tends to smooth out curvature. 
It in general lowers places with large positive curvature and increases 
places with low negative curvature. 

\paragraph{}
We actually will prove a stronger theorem. A {\bf prism graph} is a 2-sphere that 
is the suspension of a cycle graph. (In general, the join of a $k$-sphere and a $l$-sphere
is a $(k+l+1)$-sphere. The join with a $0$-sphere is a suspension. We add some remarks about
the sphere monoid in an appendix.)

\paragraph{}
A {\bf neat 3-forest cover} of a 2-sphere is a partition of $G$ into 3 forests
$A,B,C$ such that every triangle $t \in T$ intersects every of these graphs. The 
existence of a neat 3-forest cover is equivalent to a {\bf neat 4-coloring}, a 
$4$-coloring for which there are no Kempe chains. We will show:

\begin{thm}[Existence of neat forests]
Every 2-sphere $G$ different from a prism admits a neat 3 forest partition. 
\end{thm}

\paragraph{}
We can now upgrade the 4 color theorem. Of course, to prove this, we will
take the 4-color theorem for granted. 

\begin{coro}[Neat 4 color theorem]
Every 2-sphere $G$ different from a prism admits a neat 4 coloring. 
\end{coro}

Here is some indication that the ``stronger 3-forest-suffice" result is hard:

\begin{coro}
The neat 4 color theorem is equivalent to the neat 3 forest theorem.
\end{coro}

As pointed out before, it is possible prove the existence of 3 forests 
(not necessarily neat) without the 4-color theorem 
(dealing with not necessarily neat colorings). 
But so far, we could only see this as an excercise 21.4.6 in \cite{BM} 
which builds on various research of directed graphs. 

\section{Proof that 3 trees suffice}

\paragraph{}
In this section, we give the proof of the proposition provided we have a geometric
contraction map $\phi$ on a class $\mathcal{X}$ of colored $2$-spheres with Kempe loops.
more details of the evolution will be given in the next section. The main 
strategy is a {\bf descent-ascent argument}: 
reduce the area of a given colored 2-sphere $G \in \mathcal{X}$ until it is small 
enough that we can recolor it with a neat coloring. The neat coloring can then be
lifted by ascent back to the original graph. 
The smallest sphere that is reached, the octahedron, can not be recolored neatly
but before reaching it, we can recolor. 

\paragraph{}
If we can recolor some $\phi^n(G)$ neatly, then this color 
can be lifted back to $G$ so that also $G$ admits also a neat coloring. 
The {\bf strategy} is to prove that every colored 2-sphere with a Kempe loop that 
is not a prism can be compressed. 
This compression map $\phi$ reduces the area of the sphere and is possible if we have a
Kempe loop present and not a prism. There is then an explicit graph $\phi^{-1}(P_n$ that is 
reached just before entering the class of prism $P_n$.
These {\bf pre-prisms} are already large enough to be colored neatly.  
It follows that also $G$ can be recolored neatly. We have built the ladder with 
a non-neat color function (assuring the existence of Kempe loops and so 
allowing to define the map). But we climb the ladder back with the neat color function. 

\paragraph{}
Any $2$-sphere $G$ is a {\bf planar graph}. By the {\bf 4-color theorem},
there is a {\bf vertex 4-coloring} $f:V \to \{1,2,3,4\}$ which we simply call a {\bf 4-coloring}.
A {\bf colored sphere} $(G,f)$ is a $2$-sphere $G=(V,E)$ equipped with a 4-coloring $f$. 
We also equip it with a fixed orientation. If $(a,b)$ is an edge so that $f(a)=1,f(b)=2$, 
call it an edge of type $(12)$. 
The vertex coloring $f$ produces an {\bf edge coloring} $F: E \to \{A,B,C\}$ using the rules 
\begin{itemize}
\item color the edges of type $(12)$ and $(34)$ with $C$
\item color the edges of type $(13)$ and $(24)$ with $B$
\item color the edges of type $(23)$ and $(14)$ with $A$ 
\end{itemize}
Sub-graphs of $G$ with edge colors are simply called by their color name. 
Each of the sub-graphs $A,B,C$ are now graphs which are vertex colored 
by 2 colors. They are called {\bf Kempe chains} \cite{Kempe1879,Story1879,Heawood90}.
The graphs $A,B,C$ are $1$-dimensional, meaning that they are {\bf triangle free}. 

\paragraph{}
The sub-graphs $A,B,C$ are not forests in general. Closed curves in $A$ or $B$ or $C$
are called {\bf Kempe loops}. If there are no {\bf Kempe loops}, then all $A,B,C$ are forests
and we do not have to work any further as we have then already a neat 3-cover. 
We will therefore define the map $\phi$ only on colored spheres which have Kempe loops. 
Actually, the presence of a Kempe loop is important. In general, we can not shrink a colored
sphere in such a way that we have compatibility with the coloring. 

\paragraph{}
Look at the set $\mathcal{X}$ of colored oriented 2-spheres $(G,f)$ for which $f$
has at least one Kempe loop, meaning that there are closed loops of the same edge color. 
This set is not empty, as it contains any colored sphere $(G,f)$, for which $f$ is a 3-color
function (These graphs were characterized by Heawood as the class of Eulerian graphs, graphs for
which every vertex degree is even). The set also contains prisms. 
Two different Kempe loops can intersect, either with different color or the same color:
a 1-2 Kempe chain can intersect for example the 3-4 Kempe chain. 
A 3-colorable graph is an example, where every unit sphere $S(v)$ is a Kempe loop. 

\begin{propo}[Existence of a geometric evolution] 
There exists a map $\phi: \mathcal{X} \to \mathcal{X}$ that reduces the area 
of $G$, as long as long as $G$ is not an octahedron. 
The octahedron is considered a fixed point of $\phi$. If $\phi^n(G)$ is a prism, 
then $\phi^{n-1}(G)$ admits a neat coloring. 
\end{propo}

\paragraph{}
This proposition is proven in the next section and establishes that we have
a neat coloring on $G$ and so a neat forest cover establishing especially
``three trees are enough". The reason is that the evolution will reach a prism
and one step before is a colored 2-sphere $(G,f) \in \mathcal{X}$ that
allows for a 4-coloring without Kempe loops. In order to see that we have to be able
to define the map in each case, where we have a Kempe loop and $G$ is not an octahedron.
Then we have to see that we can reduce prism graphs and that so the only possible 
$G$ for which $\phi^2(G)$ is the octahedron can be recolored to have no Kempe loops.

\begin{figure}[!htpb]
\scalebox{0.5}{\includegraphics{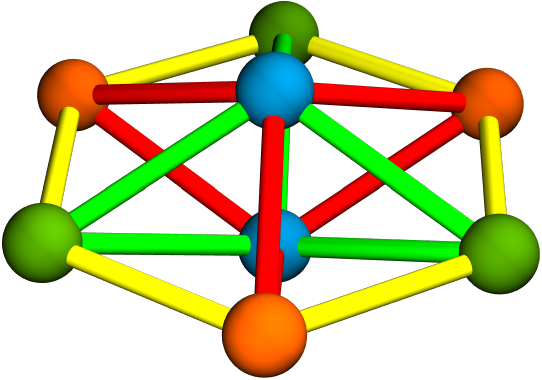}}
\scalebox{0.5}{\includegraphics{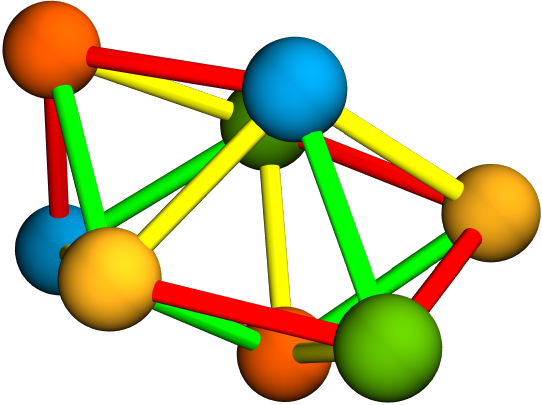}}
\caption{
This graph graph $G$ has the property that $\phi(G)$ is a prism.
This graph has colorings with Kempe chains (left) and neat colorings without
Kempe chains (right). The neat coloring can then be lifted to all the previously
obtained 2-spheres $\phi^{-n}(G)$.  }
\end{figure}

\section{The map}

\paragraph{}
We assume that $G$ is a 2-sphere that is not a prism. The map $\phi$ first remove degree-$4$ vertex 
b edge collapse, provided there are degree 4 vertices. 
This lowers the number of {\bf peak positive curvature points} and is
done by collapsing a Heawood wheel.
If no degree 4 vertices exist any more, then $\phi$ collapses a Heawood kite. 
A {\bf Heawood kite} $(H,f)$ is a colored kite graph $H$, where the coloring
is such that the two unconnected vertices having the same color. 
The kite collapse will increases curvature of some vertices and decreases curvature at 
others. The main difficulty will be to show that 
except in the prism case, but with a Kempe loop present, there always exists a Heawood kite.
Not all colored spheres have Heawood kites. But we will see
that all colored spheres with a Kempe loop must have a Heawood kite. The collapse
of a Heawood kite might introduce again degree 4 vertices, but we remain in the class of 
2-spheres. Each step $\phi$ reduces the area of $G$ by at least $2$. 

\paragraph{}
A {\bf Heawood leaf} is a vertex in a Kempe chain that has only vertex degree $1$ within
that chain. A {\bf Kempe seed} is a vertex in a Kempe chain that is isolated, not connected
to any other vertex of the chain. 
A Kempe tree either has a leaf or a seed. A $1$-dimensional graph that has no leaf and no seed
must consist of a disjoint union of loops, possibly joined by connections. 
It is $1$-dimensional network with vertex degrees all larger than $1$. 

\paragraph{}
A {\bf Heawood wheel} is a unit ball $B(v)$ of a degree-4 vertex $v$.
Degree 4 vertices are the points in $G$ with maximal curvature $1-4/6=1/3$. 
A Heawood wheel has 5 vertices and is also called a {\bf wheel graph} $W_5$. Because we can only 
use 3 colors in the unit sphere $S(v)$ of such a graph, there are by the pigeon-hole principle 
two vertices in $S(v)$ which have the same color. (it is also possible that the unit sphere $S(v)$
has only 2 colors, in which case $S(v)$ is a Kempe chain). 
These vertices $a,b$ with $f(a)=f(b)$ have to be opposite to each other.

\paragraph{}
We have a colored 2-sphere $(G,f)$, where $G=(V,E)$. 
Define the subset $V_4 \subset V$ of vertices in $G$ which have degree $4$. 
The sub-graph of $G$ generated by $V_4$ is called $G_4$. 
A {\bf linear forest} is a forest in which every tree is a linear graph or
a seed $K_1$, a graph without edges. Here is a major reason why prism graphs
are special:

\begin{lemma}
If $G$ is a 2-sphere different from a prism, then $G_4$ is a linear forest. 
\end{lemma}
\begin{proof}
If $G_4$ contains a cyclic graph $C_n$, then $G$ must be a prim graph. 
If $G_4$ contained a triangle, then it must be the octahedron (again a prism graph).
If $G_4$ contained a branch point (a star graph with 3 or more points),
then again, it would have to be the octahedron. 
\end{proof} 

\paragraph{}
Given a connected component in $G_4$, it is the center line $Q$ of a {\bf Kempe diamond}.
We call its suspension a {\bf Kempe diamond} $Q+\{a,b\}$. There are now two 
possibilities. Either $f(a)=f(b)$ or $f(a) \neq f(b)$. 
In the case $f(a)=f(b)$, we just remove the diamond by identifying $a,b$. 
All the triangles in that diamond will disappear. We still can get a degree 4
vertex like that but we have reduced the number of vertices. Now go back to 
the beginning, of this stage. After finitely many steps, we will have either
only degree 4 vertices meaning an octahedron or then reduced the number of 
degree 4 vertices by $1$. Repeat this step until no degree 1 vertices exist

\paragraph{}
In the case when $f(a) \neq f(b)$, we necessarily have that the center line 
$Q$ is a Kempe chain with only two colors. If we are not in the prism case, 
we can collapse wheels one by 
one until the center line has length $0$ or $1$. After doing so, no degree 4
vertices are left in distance 1 of the remaining center edge or point. 

\paragraph{}
All these reductions were compatible with the coloring. It can happen that a
Kempe loop is removed like for example if $G$ contains a Kempe loop of length 4. 
Collapsing that Kempe wheel removes that loop. 
The coloring $f$ of $G$ can be adapted to a coloring $\phi(f)$ of $\phi(G)$. 
We now establish that under the condition of a Kempe chain and if $G$ is different
from a prism, there must be Heawood kite.

\begin{lemma}
Given a colored sphere $(G,f)$ with a Kempe loop and where $G$ is different from a prism.
Then there is at least one Heawood kite.
\end{lemma}

\begin{proof}
A Kempe graph $C$ intersected with the 2-ball in the interior of a minimal loop
must either be empty or have a {\bf Heawood leaf}. 
\end{proof}

\begin{lemma}
The existence of a Heawood leaf in the interior of $U$ forces the existence of a 
Heawood kite. 
\end{lemma}

\begin{proof}
We can assume without loss of generality that we deal with a minimal $1-2$ 
loop $L$ and that we have a leaf ending with a vertex $v$ of color $1$ 
(the case where it ends with $2$ being similar). The unit sphere of $v$ in $C$
contains only one vertex. Call it $w$. Look at the unit sphere $S(v)$ in $G$. 
It contains $w$ with color $2$ and all other vertices have color $3,4$. So, there
are two vertices of the same color. This produces a Heawood kite. 
\end{proof} 

\paragraph{}
Under the assumption that no degree 4 vertices exist, 
we can now collapse a Heawood kite. 
A collapse reduces the vertex degrees of two of the vertices by 1.
The deformation preserves the presence of a Kempe loop. 
The conclusion is that unless are in the Octahedron case, 
there is a smaller colored sphere $(H,g)$.

\begin{lemma}[3 neat forests produce a 4 colored gem] 
Let $G$ be a 2-sphere. If $G$ has a neat 3 forest cover, then it 
has a neat 4-coloring.
\end{lemma}

\begin{proof} 
Assume $A,B,C$ are the forests giving the colors attached to the edges of $G$. 
Think of a $4$-coloring is a gauge field $V \to \mathbb{Z}_4=\{0,1,2,3\}$. 
We attach now a {\bf connection}, by assigning to every edge an element in
the subgroup of the symmetric group $S_3$ as the presentation 
$\{A,B,C, A^2=B^2=C^2=ABC=1 \}$. This is an Abelian group with 4 elements. Now start coloring
the vertices by picking a vertex $v$ and attaching to it the color $0 \in \mathbb{Z}_4$. 
Now use the connection to color all of the neighbors. This produces a coloring. 
The only problem we could encounter is if we have a closed loop of the same color
summing up to something non-zero. But this is excluded by having trees.
\end{proof} 

\paragraph{}
This article has shown that if we can 4-color a 2-sphere $G$ neatly, we can constructively 
cover $G$ with 3 forests. In applications, we want to do this fast. To speed things up, 
it helps to assume that the coloring has the property that 
every ball $B_2(x)$ of radius $2$ contains all 4 colors. If not, 
change some of the vertices $S(x)$ with the forth color. This preconditioning  breaks
already many closed Kempe chains. An extreme case with many Kempe chains
is if $G$ is Eulerian. By a theorem of Heawood, $G$ can then be colored with $3$ colors:
just start with a point, color the unit sphere, and then follow the triangles to 
color the entire graph. In the case when the range of $f$ has $3$ points, 
every unit sphere is a Kempe chain.

\begin{coro}
The construction of 3 neat forests is at least as hard as 
constructing a neat 4-coloring. 
\end{coro}
\begin{proof}
Tree data $(G,F)$ with 3 neat tree coloring $F$ produces immediately a 
$4$ vertex coloring $f$. Since the 4-coloring is hard, also the 
neat 3-tree covering is hard in the sense of complexity. 
\end{proof} 

\paragraph{}
Constructing a general 3-tree covering is computationally simpler.
We are not aware however of a procedure that upgrades any 
``forest 3 cover" to a ``neat forest 3-cover". What we have shown here is that
we can upgrade a 4-coloring to a neat 4-coloring. 

\section{Footnotes}

\paragraph{}
The arboricity of a $0$-dimensional graph by definition is $1$. Seeds are considered trees too
and a graph without edges is a collection of seeds, still a forest. 
This assumption matches the {\bf forest theorem},
telling that the number of rooted forests in a graph 
is ${\rm det}(L+1)=1$ and the number of rooted trees in a graph is ${\rm Det}(L)$, where 
$L$ is the Kirchhoff matrix of $G$ which for a zero dimensional graph is the $0$ matrix. 
See \cite{cauchybinet}. 
The arboricity for $H=K_1$ should be declared to be $1$ too even so $|E|/(|V|-1)$ technically
does not make sense. In the disconnected case, the forest arboricity (the usual book definition) 
is not the same than the tree arboricity (which we did not use as a definition).
In the connected case, these are the same notions, as indicated in the first section and in an 
appendix. The arboricity of a disjoint union of graphs is the maximum of the 
arboricity of the individual connected compoentns. 

\paragraph{}
Covering problems with trees can also focus on packing with the same "tile".
Most famous is {\bf Ringel conjecture} from 1963 which told that any tree with $n$ edges 
packs $2n+1$ times into $K_{2n+1}$. A tree of length $1$ for example fits 
$3$ times into $K_3$, a tree of length $2$ fits $5$ times into $K_5$. 
The arboricity of $K_5$ is $3$. There are 3 trees that do the covering. 
This has been proven for large $n$ \cite{RingelConjecture}. 

\paragraph{}
In chapter 5 of \cite{Diestel} there is an exercise linking {\bf chromatic number} with
{\bf arboricity}. Here it is: {\it Find a function f such that every graph of
arboricity at least f(k) has colouring number at least k, and a
function g such that every graph of colouring number at least g(k) has arboricity
at least k, for all $k \in \mathbb{N}$.} We show here that for 2-spheres,
there is an explicit relation. In general, given a coloring with $n$ colors, there is a
tree cover with $n(n-1)/2$ trees. 
If we have $k$ forests, we can color each forest with 2 color so that $2k$ is
an upper bound for the chromatic number. For general graphs there is no universal
relation between chromatic number and arboricity.

\paragraph{}
The classification of discrete 2-manifolds agrees with the classification of 
2-dimensional differentiable manifolds. In a combinatorial setting
we want to avoid the continuum. The argument that links the discrete with the 
is a geometric realization of the simplicial complex of $G$ is a 2-manifold.
The geometric realization of every ball $B(v)$ is a wheel graph which 
has a two-dimensional ball as geometric realization. 
The structure and incidence of the unit balls provides an explicit {\bf atlas} for a 
piecewise linear 2-manifold which then could be smoothed out to have a smooth
2-manifold. Finite discrete geometry is close to PL (piecewise linear) geometry. 
Topological manifolds are a much larger and much different category: 
a double suspension of a homotopy 3-sphere for example produces there a 5-sphere. 
This is not possible in the PL category or the discrete.

\paragraph{}
The question of characterizing spheres in the continuum had been pursued by Poincar\'e
already. The thread was taken in a finitist setting by Hermann Weyl during the
{\bf ``Grundlagen crisis"} in mathematics.
A combinatorial description based on cardinalities of simplices does not work because there
are complexes with the same f-vector, where one is a sphere and the other is not. 
The existence of holonomy spheres already constructed by Poincar\'e thwarted any attempts 
of a cohomological characterization of spheres Weyl's problem was solved using homotopy
or {\bf contractibility}. A fancy version is the solution of the Poincar\'e conjecture 
that characterizes $d$-spheres using homotopy for $d \geq 2$. 

\paragraph{} 
Various characterizations of spheres have been proposed in discrete, combinatorial 
settings.  The definition used here was spearheaded in digital topology by Evako \cite{Evako1994,I94a}.
Forman's discrete Morse theory \cite{FormanVectorfields,Forman2002,Forman1999} characterized spheres 
using Reeb's notions, \cite{dehnsommervillegaussbonnet}, seeing
them as manifolds admitting a Morse function with exactly 2 critical points. 
The Lusternik-Schnirelmann picture \cite{josellisknill} is to see spheres as 
manifolds of characteristic 2, meaning that it can be covered by two contractible balls.
In two dimension, Kuratowski's theorem identifies 2-spheres 
as 2-manifolds that are planar. More precisely, Whitney would have 
characterized a 2-sphere as a maximally planar 4-connected graph different from $K_4$. 
$2$-spheres are the class of graphs among planar graphs which are hardest to color. 

\paragraph{}
Spheres can also be characterized as manifolds which have trivial homotopy type and for
3-spheres as 3-manifolds which are simply connected. Poincar\'e's conjecture is now Perelmann's theorem. 
The deformation $\phi$ is a Mickey Mouse {\bf Ricci flow} and is certainly motivated
as such because we get rid of large curvature parts by removing cross wheels
and reduce negative curvature parts by collapsing kites. In higher dimensions, 
a continuum deformation argument should be able to prove a discrete Perelmann theorem.
But deformations in higher dimensions most likely leads to subtle bubble-off difficulties 
as in the continuum.

\paragraph{}
To see a hexahedron (=cube) or dodecahedron as 2-dimensional sphere with 
Euler characteristic $2$ would require us to go beyond simplicial complexes 
and use the larger category of {\bf cell complexes}, a category, where
2-dimensional cells can be attached to already present 1-dimensional cyclic graphs. This
requires care. If we would declare fr example every 4-loop in a hexahedron graph to be a ``cell",
we would have much too many faces. Historically, the continuum was a guide seeing polyhedra as
triangulations of classical spheres. 
The confusion was already evident to the time of Euler.  See \cite{lakatos,Richeson}. 
{\bf Kepler polyhedra} for example are regular polyhedra of Euler characteristic different from 2. 
An other larger category in which one can work is the category of {\bf $\Delta$-sets}, 
a pre-sheaf over a simplex category and generalizing simplicial complexes. 
This is a {\bf finite topos} and would be the ultimate finite category to work in. But graphs
are much more intuitive and accessible. 

\paragraph{}
The history of the 4-color theorem is discussed in chromatography textbooks like
\cite{SaatyKainen,Barnette,Ore,ChartrandZhang2,RobinWilson,Fisk1980}.
Kempe's proof \cite{Kempe1879} from 1879 had been refuted 1890
\cite{Heawood90} but some of his ideas survived. Already in the same issue of the 
American Journal of Mathematics, William E. Story pointed out that
Kempe's proof needs to be made rigorous \cite{Story1879}. Kempe's descent idea for the proof
allows to see quickly that one can color any 2-sphere by 5 colors. 
More geometric approaches have been spear headed by Fisk 
\cite{Fisk1977b,knillgraphcoloring,knillgraphcoloring2,knillgraphcoloring3}.
Every d-manifold can be colored by $2d+2$ colors \cite{manifoldcoloring}. Again, we have
to stress that this is a completely different set-up as the Heawood story 
Ringel and Youngs finished showing that on a genus $g=\lceil (n-3)(n-4)/12 \rceil$ 
there is a complete subgraph $K_n$ embedded. For $n=7$, this gives the famous torus case 
$g=1$ of Heawood. But this embedding of a $6$-dimensional graph in a $2$-dimensional
surface is not what we call a discrete manifold \cite{Heawood49,Ringel1974} as the 
unit sphere of a vertex in the graph $K_7$ is the $5$-dimensional $K_6$. 

\paragraph{}
For any coloring $f:V \to \mathbb{R}$, 
the Poincar\'e-Hopf index $i(v) = 1-\chi(S^-(v))$ for
the sub-graph $S^-(v)$ generated by $\{ w \in S(v), f(w)<f(v)\}$
is an integer-valued function such that $\sum_{v \in V} i(v) = \chi(G)$
an identity that holds for all finite simple graphs. 
See \cite{Glass1973,poincarehopf,parametrizedpoincarehopf,MorePoincareHopf}.
[The maximum $v$ of $f$ has $S^-_f(v)=S(v)$.
The valuation formula $\chi(A \cup B) = \chi(A)+\chi(B)-\chi(A \cap B)$ produces
$\chi(G) = \chi(G-v) + \chi(B(v)) -\chi(S(v)) = \chi(G-v)+ 1-\chi(S^-(v)) = \chi(G-v) + i(v)$,
allowing induction.]  A vertex $v \in V$ is a {\bf critical point}  of $f$
if $S^-(v)$ is not contractible. The Reeb characterization of spheres is a 
manifolds for which there is a coloring with exactly two critical points. 

\paragraph{}
The expectation of the Poincar\'e-Hopf index $i_f(v)$ over
natural probability spaces like the uniform counting measure on colorings
is the {\bf Levitt curvature} 
$K(v) = 1-\sum_{k=0} (-1)^k \frac{|f_k(S(v))|}{k+1}$ 
\cite{Levitt1992,cherngaussbonnet,valuation}.
For $K_4$-free graphs, it leads to
$1-f_0(S(v))/2+f_1(S(v))/3$ which for 2-manifolds
simplifies with $f_1(S(v))=f_0(S(v))$ to $1-f_0(S(v))/6$, a curvature
which goes back to Victor Eberhard \cite{Eberhard1891} and was used
in graph coloring arguments already at the time of Birkhoff and heavily for the
4 color theorem program \cite{Heesch}, as Gauss Bonnet $\chi(G)=2$ implies
that there must be some vertices of degree $4$ or $5$, points 
of {\bf positive curvature}. 

\paragraph{}
The integral geometric approach to 
Gauss-Bonnet $\chi(G) = \sum_{v \in V} K(v)$ leads in the continuum
for even-dimensional manifolds $M$ to the {\bf Gauss-Bonnet-Chern} result. 
See \cite{indexexpectation,colorcurvature,ConstantExpectationCurvature,DiscreteHopf2}.
This can be seen by {\bf Nash embedding} $M$ in to an ambient 
Euclidean space $E$, using that Lebesgue almost all
unit vectors $u \in E$, the function $f_u(x) = x \cdot u$ is a Morse function.
The Euler characteristic $\chi(M) = \sum_{v \in M} i(v)$ is a finite sum.
The expectation $K(v) = {\rm E}[i(v)]$ over the
probability space $\Omega=S(0)= \{ u \in E, |u|=1\}$
a rotationally invariant probability measure $\mu$ on $S(0)$ (parametrizing the
Morse functions). By a result of Hermann Weyl, this must be the 
Gauss-Bonnet-Chern curvature. 

\paragraph{}
Trees are useful data structures within networks. There are many reasons: a network 
equipped with a tree for example, gives fast search as missing loops avoids running in circles.
The data are located in the leaves of the tree, nodes with only one neighbor
Decision trees in network of decisions is an other example providing causal
strategies to reach goals located at leaf positions in the tree.
The Nash-Williams result shows that the arboricity measures a {\bf network density}. 
Indeed $|E|/(|V|-1)$ is close to $|E|/|V|$ which is half the 
{\bf average vertex degree} of the network by the {\bf Euler
handshake formula}. The arboricity is essentially the maximal density of subgraphs of $G$. 

\paragraph{}
If we know a finite set of rooted trees covering the graph, we know the entire
network. The reason is simple: given two nodes $v,w$ in the network, then
$v,w$ are connected if and only if in one of the rooted trees, the distance between $v$
and $w$ is $1$. The {\bf arboricity} of a graph
the minimal number of trees which are needed to cover all connections.
Within a network it can serve as some sort of dimension. There are lots of interesting
questions, like how many trees there can be for a neat coloring of a 2-sphere. 

\paragraph{}
We have already called a graph $G=(V,D)$ to be {\bf contractible} 
if there is a vertex $v \in V$ such that the unit sphere 
$S(v)$ as well as the graph generated by $V \setminus \{v\}$
are both contractible. The {\bf Lusternik-Schnirelmannn category} ${\rm cat}(G)$ 
is the minimal number of contractible sub-graphs covering $G$. One has
${\rm cat}(G) \leq {\rm arb}(G)$ in the positive dimensional connected graph case
because every tree is contractible. We will see that for a 2-sphere ${\rm cat}(G)=2$ 
and always ${\rm arb}(G)=3$.  With the arboricity of a $0$-dimensional graph 
defined to be 1, the inequality ${\rm cat}(G) \leq {\rm arb}(G)$ only works 
for connected graphs. 

\paragraph{}
Whitney once characterized 2-spheres using duality. 
There is a theorem of Karl von Staudt \cite{Staudt1847} 
telling that $G$ and $\hat{G}$ 
have the same number of spanning trees, because every spanning tree 
for $G$ produces a spanning tree in $\hat{G}$. The dual graph however
is only 1-dimensional and has smaller arboricity than $G$ if $G$ is a
2-sphere. See \cite{KnillTorsion}. 

\paragraph{}
Prims pointed towards the {\bf arithmetic of graphs}. 
The Zykov monoid can be group completed and
together with the large product (introduced by Sabidussi \cite{Sabidussi}) 
gies the {\bf Sabidussi ring}. It is isomorphic to the {\bf Shannon ring} in which 
the disjoint union is group completed and where the 
Shannon multiplication (strong multiplication) is used \cite{Shannon1956}. 
In the Shannon ring, the connected components are the {\bf additive co-prodoct primes}. 
In the Sabidussi ring, the graphs which are not the join of two non-zero graphs
are the {\bf additive Zykov primes}. Spheres form a sub-monoid of the Zykov monoid and
most spheres are prime. In dimension $2$ there are very few non-primes: they are the prisms.  
\cite{ArithmeticGraphs,RemarksArithmeticGraphs,numbersandgraphs}.

\paragraph{}
For Dehn-Sommerville relations for any $d$-sphere,
see \cite{Sommerville1929}. See also \cite{DehnSommerville} for a Gauss-Bonnet approach
or \cite{dehnsommervillegaussbonnet}.
They are quantities which are both invariant and explode under Barycentric refinements.
The only way that this can happen is that they are $0$. They can be 
obtained from the eigenvectors of the Barycentric refinement operator. 
See \cite{KnillBarycentric, KnillBarycentric2}. See \cite{TreeForest} for the tree
forest ratio. The {\bf Nash-Williams} density $f_1/(f_0-1)$ measures some sort of density
of the network, similarly as the average simplex cardinality \cite{AverageSimplexCardinality}
or other functionals \cite{KnillWolframDemo1,KnillFunctional}. This could lead to interesting 
variational problems like which $d$-manifolds produce the largest Nash-Williams density. 
We hope to work on this a bit more in the future. 

\paragraph{}
Does cohomology determine the arboricity for manifolds? In other words, if two 
graphs have the same Betti vector $(b_0,b_1, \dots)$ and the same dimension
does this imply that the arboricity is the same? This is not true: 
the arboricity of a sufficiently large $d$-sphere
can be estimated by the Perron-Frobenius eigenvector of the Barycentric refinement 
operator and must be at least to a constant $c_d$ that grows exponentially.
But the arboricity of a suspension of $G$ is less or equal than the arboricity 
of $G$ plus $2$. In higher dimension, the {\bf arboricity spectrum} 
gets wider and wider within the same class of graphs. There are $d$-spheres
that can be colored with $2d$ trees and there are spheres with arboricity 
growing like $c_d = 1.70948^d$. We hope to explore this a bit more in the future.
Especially interesting is the question whether the constants $c_d$ define the 
upper bound. We currently believe that this is the case. 

\paragraph{}
We believe that the arboricity of any 2-manifold is maximally 4. We have no proof of that
but it would follow if the arboricity of a $d$-manifold were the smallest integer larger
than the Barycentric arboricity constant $c_d$ which is in the case $d=2$ equal to $3$. 
There are examples like tori or Klein bottles with arboricity 4. 
In particular, we believe that the arboricity of any 3-manifold is smaller or equal than
the smallest integer larger than $13/2=6/5$ which is $7$. We know using Barycentric refinement
that there are 3-spheres with arboricity $7$. The suspension of the octahedron, the smallest
3-sphere has Nash-Williams ration $24/7$ and so arboricity $4$. We see that there are 3-spheres
of arboricity $4,5,6,7$. This is completely different than in the 2-dimensional case considered
here, where the three theorem assures that the arboricity of a 2-sphere is locked to be $3$. 
For 5-spheres, the range of possible arboricity is $5,6, \dots, 13$. The lower bound grows
linearly with the dimension, while the upper bound grows exponentially with the dimension. 

\paragraph{}
If $c_d$ is integer, that it can be by 1 larger. For $d=2$, where $c_d=3$, this
happens for 2 tori as in dimension $2$ the Barycentric refinement limit is 
$3$ and the arboricity can be $4$. 
The first constants are $\left\{ c_1,c_2,c_3,\dots, c_{10} \right\}$ are
$$\left\{ 1,3,\frac{13}{2},\frac{25}{2},\frac{751}{33},\frac{1813}{45},\frac{3087965}{43956},
  \frac{6635137}{54628},\frac{488084587476}{2335501595},\frac{200710991888}{559734735} 
  \right\} \; . $$ 
This suggests that the correct maximal arboricity numbers for $d$-spheres in the case
$d=0,1,\dots, 10$ 
are 
$$  \{1, 2, 3, 7, 13, 23, 41, 71, 122, 209, 359 \} $$
We do not know whether there are integer $c_d$ besides $c_1,c_2$. There are lots of open
questions. Are there 3-spheres of arboricity 8 for example. We believe this is not possible. 

\section*{Appendix: Arboricity}

\paragraph{}
The current public material about the arboricity of graphs 
is rather sparse. The functional is mentioned in \cite{HararyGraphTheory,BM}. 
Most graph theory textbooks do not cover it. Both in \cite{HararyGraphTheory}
and \cite{BM}, it is defined as the minimal number of forests partitioning
the graph. This has been mentioned in the introduction. Lets elaborate about
it a bit more. 

\paragraph{}
A {\bf tree} is a triangle free graph which is contractible. The smallest tree is $K_1$
and also called a {\bf seed}. Note that the zero graph is not a tree. 
Every tree has Euler characteristic $\chi(G) =|V|-|E|=1$. 
A {\bf forest} is a triangle free graph for which every connected component is a tree.
A forest has Euler characteristic $\chi(G) = b_0$, which is the number of connected
components of the graph. A forest can be contracted to a $0$ dimensional manifold.
A {\bf spanning tree} in a connected graph $G$ is a subgraph that is a tree
and which has the same vertex set than $G$. A {\bf spanning forest} 
is a forest which has the same vertex set than $G$. 

\begin{lemma}
If $G$ is connected, the following numbers are the same: \\
a) The arboricity $T_1(G)$, the minimal number of forests partitioning $G$.  \\
b) The minimal number $T_2(G)$ of trees covering $G$. \\
c) The minimal number $T_3(G)$ of spanning trees covering $G$
\end{lemma}
\begin{proof} 
This follows from the fact that every spanning tree is a  tree 
and every tree is a forest and that every forest can be completed
to a spanning tree in a connected graph. 
\end{proof}

\paragraph{}
If $G$ is not connected, the numbers $T_1(G),T_2(G)$ are not the same. For a forest $G$ with 
2 trees for example, we have $T_1(G)=1$ while $T_2(G)=2$. 
A related interesting concept is {\bf linear arboricity} $T_4 \geq T_2(G)$ which asks 
how many linear graphs partition the graph. 
A lower bound is $[\Delta/2]$ (where [x] is the largest integer smaller or equal to x), 
where $\Delta(G)$ is the maximal vertex degree.
The {\bf linear arboricity conjecture} asks whether $[(\Delta+1)/2]$ 
is an upper bound for linear arboricity. If a regular graph 
$G$ of even vertex degree $\Delta$ gets a vertex removed, one gets a graph 
$H$ whose linear arboricity is $\Delta/2$ if and only if the graph $G$ 
has a Hamiltonian decomposition. 

\paragraph{}
A graph is called {\bf contractible}, if there exists $v \in V(G)$ such 
that $S(v)$ and $G \setminus v$ are both contractible. 
Contractible graphs are connected. Trees are contractible, forests that 
are not trees are not contractible. 
The {\bf Lusternik-Schnirelmann category} (or simply category) of a graph 
is the minimal number of contractible graphs that cover the graph. 
Since trees are contractible, the category in general is smaller or equal 
than the number of trees covering $G$. 
For a 2-sphere $G$, the category of $G$ is $2$, while the arboricity of $G$ 
is $3$. 

\paragraph{}
If $G$ is a graph with $f$-vector $f=(f_0,f_1, \dots, f_d)$, where $f_k$ is the 
number of $k$ dimensional simplices (=cliques=faces) $K_{k+1}$ in $G$. 
The graph $G_1$ in which the complete subgraphs of $G$ are the
vertices and where two are connected if one is contained in the other
is called the {\bf Barycentric refinement} of $G$. If $G$ was contractible
then $G_1$ is contractible. If $G$ is a $d$-sphere, then $G_1$ is a $d$-sphere. 
The graph $G_n=(G_{n-1})_1$ is the $n$'th Barycentric refinement. 

\paragraph{}
The {\bf Barycentric refinement matrix} $A$ is a $(d+1) \times (d+1)$ matrix
with entries
$$  A_{ij} = i! S(j,i)  \; , $$
where $S(j,i)$ are the {\bf Stirling numbers} of the second kind.
The vector $Af$ is the f-vector of $G_1$. Define 
$$ c_d = \lim_{n \to \infty} E(G_n)/(V(G_n)-1) 
       = \lim_{n \to \infty} f_1(G_n)/f_0(G_n) \; . $$
Since the normalized $f$-vectors $f(G_n)/||f(G_n)||_2$ ($|| \cdot ||_2$ is the Euclidean norm), 
converge to the {\bf Perron-Frobenius vector} we can compute the limits $c_d$.
They grow like $1.70948^d$. From the Nash-Williams formula, we see
that the arboricity of $d$-spheres grows exponentially in $d$. 

\begin{propo}
a) If $H$ is a sub-graph of $G$ then ${\rm arb}(H) \leq {\rm arb}(G)$.  \\
b) The arboricity of the disjoint sum of two graphs $H,G$ 
   is ${\rm max}({\rm arb}(H),{\rm arb}(H))$. \\
c) There are graphs with a $K_{n+1}$ subgraph with arboricity $\geq c_n$ \\
d) The arboricity of $K_n$ is ${n/2}$ so that any graph with n vertices has
   arboricity $\leq {n/2}$.  \\
e) For a connected graph, the Lusternik Schnirelmann 
   category of $G$ is a lower bound for the arboricity.
\end{propo}
\begin{proof} 
a) This follows directly from Nash Williams. The inequality also can be derived from the definition
   and is used in proofs of \cite{HararyGraphTheory}.  \\
b) This follows from the definition and that forests can be disconnected. \\
c) This follows from the Barycentric refinement picture and Nash-Williams. \\
d) For a complete graph $K_{2n}$ we can get the forests explicitly. 
   See page 92 in \cite{HararyGraphTheory}.  \\
e) This follows because for connected graphs, we can use the tree cover picture and because trees
   are contractible.
\end{proof} 

\section*{Appendix: The sphere monoid}

\paragraph{}
The {\bf Zykov join} $G \oplus H$ of two graphs $G,H$ has $V(G) \cup V(H)$ as vertex set 
$E(G) \cup E(H) \cup \{ (a,b), a \in V(G), b \in V(H) \}$ as edge set. 
It is dual to the disjoint union of graphs 
$\overline{\overline{G} +\overline{H}}$, where $\overline{H}$ is the graph complement. 
The Zykov join $G+S_0$ of $G$ with a $0$-sphere is called the {\bf suspension} of $G$. 
The suspension of a $d$-sphere is a $(d+1)$-sphere. 

\paragraph{}
The union of all $k$-spheres, where $k \geq (-1)$ forms a sub-monoid of the 
Zykov monoid of all graphs. The {\bf Dehn-Sommerville monoid} is 
a monoid strictly containing the sphere monoid but where the unit
spheres are required to have Euler characteristic of the corresponding
spheres but are Dehn-Sommerville manifolds of a dimension lower. 
Examples are disjoint unions of odd dimensional spheres
or disjoint copies of two even dimensional spheres
glued along north and south pole so that the unit spheres are either spheres
or unions of spheres. 

\begin{lemma}
If $G$ is a $p$-sphere and $H$ is a $q$-sphere, 
then $G \oplus H$ is a $(p+q+1)$-sphere. The empty graph $0$
is the {\bf zero element}.
\end{lemma}

\begin{proof}
Use induction and note that for $v \in V(G)$, the unit sphere is 
$S_{G \oplus H}(v) = S_G(v) \oplus H$ which by induction is a $(p+q)$-sphere.  
Similarly the unit sphere for $v \in V(H)$ is a $p+q$ sphere. 
\end{proof} 

\paragraph{}
A graph is a {\bf Zykov prime} if it can not be written as 
$G=H \oplus K$, where $H,K$ are both non-zero graphs. 
A small observation about {\bf small non-prime spheres}: 

\begin{lemma}
a) The only non-prime 1-sphere is $C_4$.  \\
b) The only non-prime 2-sphere is a prism graph $C_n \oplus S_0$. \\
c) For every $n \geq 6$ there exists exactly one non-prime $2$-sphere with $n$ vertices. \\
d) For every $n \geq 8$ the number of non-prime $3$-spheres with $n$ vertices is $2$ plus the 
   number of 2-spheres with $(n-2)$ vertices. 
\end{lemma}

\paragraph{}
The fact that the only non-prime 2-spheres are prisms makes them special. 
If we make an edge refinement of a prism along an edge in the equator,
we get a larger prism. If we make an edge refinement along an edge hitting
one of the poles, we get an {\bf almost prism}, a graph which allows a 
4 coloring without Kempe loops.  

\bibliographystyle{plain}

\end{document}